\documentclass[12pt,pdftex,noinfoline]{imsart}

\usepackage[usenames,dvipsnames]{xcolor}
\usepackage{natbib}
\bibpunct{(}{)}{;}{a}{,}{,}

\usepackage[pdftex,pdfborderstyle={/S/U/W 1},colorlinks]{hyperref}
\RequirePackage[OT1]{fontenc}
\RequirePackage{amsthm,amsmath,amsfonts}
\RequirePackage{hypernat}
\usepackage{fullpage}
\usepackage{graphicx}
\usepackage{hyperref}
\usepackage{amsthm, amssymb, amsmath}
\usepackage{url}
\usepackage[english]{babel}
\usepackage{times}
\usepackage[T1]{fontenc}
\usepackage{bbm}
\usepackage{xspace}
\usepackage{rotating,multirow}
\usepackage{enumerate}
\usepackage{tikz}

\def\interleave{|\kern-.25ex|\kern-.25ex|}
\def\interleavesub{|\kern-.15ex|\kern-.15ex|}

\newcommand{\nNorm}[1]{\left|\kern-.25ex\left|\kern-.25ex\left| {#1}\right|\kern-.25ex\right|\kern-.25ex\right|}

\newcommand{\E}{{\mathbb E}}

\newcommand{\eps}{{\epsilon}}

\newcommand{\R}{{\mathbb R}}

\renewcommand{\P}{{\mathbb P}}

\newcommand{\C}{{\mathcal{C}}}

\newcommand{\I}{{\mathcal{I}}}

\newcommand{\K}{{\mathcal{K}}}

\newcommand{\F}{{\cal F}}

\def\min{\mathop{\text{\rm min}}}
\def\max{\mathop{\text{\rm max}}}

\def\sup{\mathop{\text{\rm sup}}}

\numberwithin{equation}{section}
\theoremstyle{plain}
\newtheorem{theorem}{Theorem}[section]

\newtheorem{proposition}{Proposition}[section]
\newtheorem{lemma}{Lemma}[section]
\newtheoremstyle{remark}{\topsep}{\topsep}%
     {\normalfont}
     {}           
     {\bfseries}  
     {.}          
     {.5em}       
     {\thmname{#1}\thmnumber{ #2}\thmnote{ #3}}
\theoremstyle{remark}
\newtheorem{remark}{Remark}[section]

\long\def\comment#1{}

\def\P{{\mathbb P}}
\def\E{{\mathbb E}}
\def\supp{\mathop{\text{supp}\kern.2ex}}
\def\argmin{\mathop{\text{\rm arg\,min}}}

\def\argmax{\mathop{\text{\rm arg\,max}}}
\let\hat\widehat
\let\tilde\widetilde

\let\hat\widehat
\let\tilde\widetilde

\def\1{{(1)}}
\def\2{{(2)}}

\long\def\comment#1{}

\long\def\comment#1{}

\def\P{{\mathbb P}}
\def\E{{\mathbb E}}

\def\supp{\mathop{\text{supp}\kern.2ex}}
\def\argmin{\mathop{\text{\rm arg\,min}}}

\def\argmax{\mathop{\text{\rm arg\,max}}}
\let\tilde\widetilde

\let\hat\widehat
\let\tilde\widetilde

\def\1{{(1)}}
\def\2{{(2)}}

\long\def\comment#1{}

\def\threebars{\mbox{$|\kern-.25ex|\kern-.25ex|$}}

\def\K{{K}}

\def\F{\mathbb{F}}

\begin{document}

\hypersetup{citecolor=MidnightBlue}
\hypersetup{linkcolor=Black}
\hypersetup{urlcolor=MidnightBlue}

\begin{frontmatter}

\mbox{}
\vskip.25in
\centerline{\Large\bf An Improved Global Risk Bound in Concave Regression}
\runtitle{Risk Bound in Concave Regression}

\begin{aug}
\author{\fnms{Sabyasachi} \snm{Chatterjee}\ead[label=e1]{sabyasachi@galton.uchicago.edu}}
\address{
\vskip1pt
\begin{tabular}{c}
Department of Statistics \\
The University of Chicago 
\end{tabular}
\\[10pt]
\today\\[5pt]
}
\end{aug}

\maketitle 

\begin{abstract}
A new risk bound is presented for the problem of convex/concave function estimation, using the least squares estimator. The best known risk bound, as had appeared in~\citet{GSvex}, scaled like $\log(en) n^{-4/5}$ under the mean squared error loss, up to a constant factor. The authors in~\cite{GSvex} had conjectured that the logarithmic term may be an artifact of their proof. We show that indeed the logarithmic term is unnecessary and prove a risk bound which scales like $n^{-4/5}$ up to constant factors. Our proof technique has one extra peeling step than in a usual chaining type argument. Our risk bound holds in expectation as well as with high probability and also extends to the case of model misspecification, where the true function may not be concave.
\end{abstract}

\vskip15pt
\end{frontmatter}

\section{Introduction}
In this paper we consider the problem of estimating a concave function in a standard additive noise regression model. We observe $y = (y_1,\dots,y_n)$ which are noisy evaluations of an underlying concave function defined on the unit interval $[0,1]$ on sorted design points $x = (x_1 < \dots < x_n) \in [0,1]^n$  as follows;
\begin{equation}
y_i = f(x_i) + z_i
\end{equation} 
with $z = (z_1,\dots,z_n)$ being independent $N(0,\sigma^2)$ error for some unknown $\sigma$. The task of estimating the regression function $f$ when knowing that the underlying $f$ is concave is known as the concave regression problem. This problem has a relatively long history.

Interestingly, the problem of concave regression was discussed in the literature by~\citet{Hildreth54} even before the problem of monotone regression (see \citet{RWD88}, \citet{ayer1955empirical}), which is perhaps the most well studied problem in shape constrained estimation. \cite{Hildreth54} discusses some applications in economics where concavity arises naturally, such as in production functions and utility functions. A very natural estimator of a concave regression function is the constrained least squares estimator (LSE) which is defined as follows: 
\begin{equation}\label{eqn:concave_lse}
\hat{f} = \argmin_{f: [0,1] \rightarrow \R, f concave} \sum_{i = 1}^{n} (y_i - f(x_i))^2.
\end{equation}
As is clear, the LSE $\hat{f}$ is not unique but is uniquely defined at the design points. This estimator was proposed in~\cite{Hildreth54} along with quadratic programming methods for solving \eqref{eqn:concave_lse}. The first statistical analysis for the LSE was carried out by~\citet{HanPled76} where the consistency of the LSE was established under the supremum loss over a compact interval. The problem of estimation of the concave function under a local loss function, that is the problem of estimating the regression function $f$ at a given point $x_0,$ has also received attention in the literature. Notable papers are~\citet{Mammen91},~\citet{GroeneboomJongbloedWellner2001a} where the asymptotic distribution of the LSE at the point $x_0$ was studied. The first work which studied rates of convergence of the LSE under a global loss appeared in~\cite{DuembgenEtAl04}. The authors there showed that under the supremum loss over a compact interval containing the design points, the rates of convergence varies from $(\frac{\log n}{n})^{1/3}$ to $(\frac{\log n}{n})^{2/5}$ depending on the smoothness of the underlying concave function. For instance, if the regression function is also twice differentiable the rate of convergence under the supremum loss is $(\frac{\log n}{n})^{2/5}.$

In the most recent work on concave regression under a global loss function,~\citet{GSvex} considered the following sequence formulation of the concave regression problem. Assuming that the design points are equally spaced, note that the mean regression vector $(f(x_1),\dots,f(x_n))$ lies in a polyhedral cone (defined by the concavity constraint on $f$) $K_n \subset \R^n$, $n \geq 3$ defined as follows:
\begin{equation*}
K_n = \{\theta \in \R^n: \theta_{i} - \theta_{i - 1} \geq \theta_{i + 1} - \theta_{i} \: \text{whenever}\: 2 \leq i \leq n - 1\}.
\end{equation*}
Then it is clear that if $x_1 < x_2 \dots < x_n$ are equispaced design points in $\R$ then one can also write $$K_n = \{\theta \in \R^n: \theta = (f(x_1),\dots,f(x_n))\:\:\text{for some concave function}\:\: f: \R \rightarrow \R\}.$$ 
The problem of estimating the entire concave function in the sequence model then becomes the problem of estimating $\theta^* \in K_n$ from observations
\begin{equation}\label{model}
y_i = \theta^*_i + z_i \:\:\:\:\forall 1 \leq i \leq n
\end{equation}
where $z_i$ are i.i.d $N(0,\sigma^2)$ random variables. 

In this paper, we study this sequence formulation of the concave regression problem, assuming equally spaced design points. Specifically, we study the risk properties of the Least Squares Estimator $\hat{\theta}$ defined as 
\begin{equation}
\hat{\theta} = \argmin_{\theta \in K_n} \|y - \theta\|^2,
\end{equation}
where $\|\cdot\|$ refers to the usual Euclidean norm on $\R^n.$ We are interested in studying the risk behaviour of $\hat{\theta}$ under the natural mean squared error loss function defined as
\begin{equation}
R(\hat{\theta},\theta^*) = \E \frac{1}{n} \|\hat{\theta} - \theta^*\|^2,
\end{equation}
where the expectation is taken under the distribution of $y$ as given in~\eqref{model}.

To describe the results of \cite{GSvex} under the risk function $R(\hat{\theta},\theta^*),$ we first have to set up some notation. Let $L$ be the subspace of $\R^n$ spanned by the constant vector $(1,\dots,1)$ and the vector $(1,2,\dots,n).$ In words, $L$ is the subspace of $n$ dimensional affine sequences. Let $P_{L}$ denote the orthogonal projection matrix to the subspace $L.$ Define $G_{\theta^*} = \max\{1,\frac{1}{n} \|(I - P_{L})\theta^*\|^2\}.$ Then Theorem 2.2 in~\cite{GSvex} shows that for any $\theta^* \in K_n$  there exists a universal constant $C$ such that whenever the sample size $n \geq C \frac{\sigma^2}{(G_{\theta^*})^2} (\log \frac{en}{2})^{5/4},$ we have
\begin{equation}\label{gb}
R(\hat{\theta},\theta^*) \leq C (\log \frac{en}{2}) \big(\frac{\sigma^2 \sqrt{G_{\theta^*}}}{n}\big)^{4/5}.
\end{equation}

To interpret the theorem, it is instructive to think of a regime where $G(\theta^*)$ stays bounded as $n$ grows. For example, this would be the case if $\theta^*$ are evaluations on a grid of a concave function $f:\R \rightarrow \R$ with $\int_{0}^{\infty} f^2(x) dx < \infty.$ In this case, the bound in~\ref{gb} yields $R(\hat{\theta},\theta^*) \leq (\log \frac{en}{2}) n^{-4/5}$ upto a multiplicative constant factor.   \cite{GSvex} conjectured that there should be no logarithmic term in the upper bound and perhaps the logarithmic term is an artifact of their proof. The goal of this paper is to show that one can derive a risk bound in the same setting as in~\eqref{gb}, but without the logarithmic factor.

The paper is organized as follows. In Section 2 we present our main result on the global risk bound, comment on its interpretations, and sketch the proof ideas. In Section 3, we present the proof of our main result. In Section 4, we state and prove an extension of our main result in the case of model misspecification.

\section{Main Result}
We now describe the main result of this paper and discuss its consequences. For any vector $\theta \in \R^n$ let us define $V(\theta) = \max_{1 \leq i \leq n} \theta_i - \min_{1 \leq i \leq n} \theta_i.$
\begin{theorem}\label{main}
Fix any positive integer $n.$ Fix any $\theta^* \in K_n.$ Fix any $x > 0.$ The following upper bound on the risk holds for a universal constant $C$ with probability greater than $1 - \exp(-x) - \exp(-x^2/16):$
\begin{align}\label{highp}
\begin{split}
\frac{1}{n} \|\hat{\theta} - \theta^*\|^2 \leq \:\:\frac{1}{n} \sigma^2(2 + 17x) \:+ C \sigma^{8/5} \big(V\big((I - P_{L}) \theta^*\big) + \sigma \big)^{2/5} n^{-4/5}.
\end{split}
\end{align}
Moreover, the above high probability risk bound also immediately implies a risk bound in expectation for a universal constant $C$:
\begin{equation}\label{myexpbd}
R(\hat{\theta},\theta^*) \leq C \sigma^{8/5} \big(V\big((I - P_{L}) \theta^*\big) + \sigma \big)^{2/5} n^{-4/5} + \frac{2 \sigma^2}{n}.
\end{equation}
\end{theorem}

Theorem~\ref{main} proves a high probability bound for $\frac{1}{n} \|\hat{\theta} - \theta^*\|^2$ which then immediately leads to a bound in expectation. We show in Section~\ref{secmisspec} that we can also extend our risk bound to the case of model misspecification. By model misspecification, we mean the case when the true underlying sequence $\theta^*$ is not concave. In this case, we prove that the risk of the LSE $R(\hat{\theta},\theta^*)$ is at most the squared Euclidean distance of $\theta^*$ to $K_n$ divided by $n,$ plus a term which goes down to zero at the rate $n^{-4/5}.$

There are two essential differences between the risk bound~\eqref{myexpbd} and the bound~\eqref{gb}. The first difference is that we have no logarithmic terms in our risk bound. In this way, we improve over the best known risk bound in concave/convex regression. The second difference is that the term $G(\theta^*)$ is replaced by $V\big((I - P_{L}) \theta^*\big).$ In general, $V\big((I - P_{L}) \theta^*\big)$ could be larger than $G(\theta^*)$ but stays bounded by a constant if $\theta^*$ is a vector of evaluations of a concave function $f$ on a grid.

Our analysis actually carries through for slightly more general design points than just equally spaced ones; see Remark~\ref{design} for more explanation. Coming to the issue of optimality of the LSE, the rate $n^{-4/5}$ is known to be minimax rate optimal even if one considers the parameter space to be an appropriate local ball around a convex function with positive curvature everywhere; see Theorem 5.1 in~\cite{GSvex} for the precise result. Apart from minimax rate optimality, a further motivation to study the LSE in concave regression is that it was shown to have automatic adaptation properties first in~\cite{GSvex} and subsequently in~\cite{chatterjee2015risk} and~\cite{bellec2015sharp}. Specifically, it was shown that the risk $R(\hat{\theta},\theta^*)$ scaled like $\frac{k}{n} \log n$ for piecewise linear concave sequences $\theta^*$ with $k$ pieces. In this paper though, our focus is not on the automatic adaptive properties of the LSE, but rather in proving an improved worst case performance.

It is worthwhile to point out that we do not assume any smoothness assumptions on the underlying concave function. Considering the result of~\citet{DuembgenEtAl04}, their loss function is the supremum loss function over an appropriate compact interval. Notwithstanding some details, the mean squared error loss function we consider is naturally bounded by the square of the supremum loss function that is considered in~\citet{DuembgenEtAl04}. Hence their risk bound directly applied to the mean squared error loss scales like $(\frac{\log n}{n})^{4/5}$ only when the underlying concave function is twice differentiable. Our analysis shows that when the design points are equally spaced, we do not need any smoothness assumptions on the underlying concave function to obtain the $n^{-4/5}$ rate in addition to showing that the logarithmic factor is unnecessary.

\subsection{Proof Sketch for Theorem \ref{main}}
The goal of this subsection is to provide a high level overview of the method of proof of Theorem~\ref{main}. We take the standard empirical process based approach in proving our risk bound. Specifically we take the recipe proposed by~\citet{Chat14} which says that a key ingredient in proving risk bounds for the LSE is to control an expected Gaussian suprema term described as follows. Let $\langle a,b \rangle$ denote the usual inner product between any two vectors $a,b.$ For any $\theta^* \in \R^n$ define the function $f_{\theta^*}: \R_{+} \rightarrow \R$ as follows:
\begin{equation}\label{gausup}
f_{\theta^*}(t) = \E \sup_{\theta \in K_n: \|\theta - \theta^*\| \leq t} \langle z, \theta - \theta^* \rangle
\end{equation}
where $z$ is a random Gaussian vector with each entry being independent, has mean zero and variance $\sigma^2.$~\cite{Chat14} also shows that if one obtains $s > 0$ such that $f_{\theta^*}(s) \leq \frac{s^2}{2}$ then essentially one gets the risk bound $R(\hat{\theta},\theta^*) \leq s^2/n$ up to constant factors. Precise statements are given in later sections. Therefore, it suffices to tightly upper bound the function $f_{\theta^*}$ and this paper basically gives a new way of upper bounding the function $f_{\theta^*}(t)$ for all $t \geq 0.$

Since $f_{\theta^*}(t)$ is an expected Gaussian maxima, a standard tool in empirical process theory to upper bound $f_{\theta^*}(t)$ is to use the Dudley's entropy integral bound which requires good estimates of the covering number of the set $K_n \cap B(\theta^*,t)$ where $B(\theta^*,t)$ refers to the Euclidean ball of radius $t$ centred at $\theta^*.$ Tight estimates (without logarithmic factors) of the covering numbers for the space of bounded convex sequences are available in the literature, see~\citet{Dryanov}. Specifically the result of~\cite{Dryanov} gives us tight upper bounds of the covering number for the space $\{\theta \in K_n: \max_{1 \leq i \leq n} |\theta_i| \leq B\}$ for any constant $B.$ But we require covering numbers for $K_n$ intersected with a Euclidean ball of radius $t.$ This was done in~\cite{GSvex} by applying the basic result of~\cite{Dryanov} in appropriate subintervals. This approach gives rise to logarithmic factors in the risk bound as had appeared in~\cite{GSvex}.

Our approach is to do one more step of refining the function $f_{\theta^*}.$ Essentially for any $\theta \in K_n \cap B(\theta^*,t)$ we define its truncated version $\theta^{'} \in K^{'}_n$ where the set $K^{'}_n \subset \R^n$ is not a much larger set than $K_n$ and hence has comparable metric entropy. Note that one can write 
\begin{equation}\label{eqbrkintro}
\E \sup_{\theta \in K_n} \langle z, \theta - \theta^{*} \rangle \leq \E \sup_{\theta \in K_n} \langle z, \theta - \theta^{'} \rangle + \E \sup_{\theta^{'}} \langle z, \theta^{'} - \theta^{*} \rangle.
\end{equation}

We control the two terms on the right side of the above inequality separately. We show it is possible to define the truncation $\theta^{'}$ of $\theta$ such that it satisfies two critical properties. The first property is that the first term on the right side of~\eqref{eqbrkintro} is upper bounded by $\frac{t^2}{4}.$ Since we finally have to compare $f_{\theta^*}(t)$ with $\frac{t^2}{2}$ an extra factor of $\frac{t^2}{4}$ can only affect the risk bound upto constants. Now we are left with the task of upper bounding the second term in the right side of~\eqref{eqbrkintro}. The second critical property that $\theta^{'}$ satisfies is that $\theta^{'}$ is bounded entrywise by $C\big(V(\theta^*) + \sigma\big)$ where $C$ is a universal constant. This means that $\theta^{'}$ is bounded by a constant factor. Hence the second term in the right side of~\eqref{eqbrkintro} can now be controlled by a direct application of Dudley's entropy integral bound. This only requires getting tight estimates of the covering number of bounded sequences in $K^{'}_n.$ We show that this covering number is very similar to the one given in~\cite{Dryanov} (and hence has no logarithmic factors) as $K^{'}_n$ is not much bigger than $K_n.$ In this way our extra refinement step enables us to save a logarithmic factor of $n.$

The main crux of this paper lies in defining $\theta^{'}$ as described in the previous paragraph. This is done in Lemma~\ref{key1} which is perhaps the most important step in our entire argument. In Lemma~\ref{key1} we actually assume our underlying concave sequence $\theta^*$ is also monotonic. Since a concave sequence always first increases and then decreases, it actually suffices to analyze the case when $\theta^*$ is monotonic concave. We use the monotonicity of $\theta^*$ crucially in coming up with a definition of the truncation $\theta^{'}$ of $\theta,$ satisfying the critical properties as explained in the previous paragraph. 
 
\section{Proof of Main Result}

\subsection{\textbf{Proof of Theorem~\ref{main}}}\hspace*{\fill}\\
The goal of this section is to state and prove Theorem~\ref{main} which improves the best known risk bound in concave regression by removing a logarithmic term. Before starting to prove the above theorem, we first go through some background results. By now it is known that a key ingredient in proving risk bounds for the LSE is to control the expected Gaussian suprema function $f_{\theta^*}$ as defined in~\eqref{gausup}. For any $\theta^* \in K_n$ it was actually shown in Theorem 1.1 in~\citet{Chat14} that the loss term $\|\hat{\theta} - \theta^*\|$ concentrates around a deterministic value $t_{\theta^*} = \argmax_{t \geq 0} f_{\theta^*}(t) - \frac{t^2}{2}.$ Another result providing upper bounds on the loss term $\|\hat{\theta} - \theta^*\|$ in terms of the function $f_{\theta^*}$ has been given in Theorem 12 in~\citet{bellec2015sharp} which we use as our starting point. Before describing this result, we claim that since the projection onto the cone $K_n$ is a sum of projection onto the subspace $L$ and the projection onto the cone $K_n \cap L^{\perp}$ it almost suffices to study the risk of the least squares estimator constrained to the cone $K_n \cap L^{\perp}$ denoted by $\hat{\theta}_{K_n \cap  L^{\perp}}.$ Specifically we mean 
\begin{equation*}
\hat{\theta}_{K_n \cap  L^{\perp}} = \argmin_{\theta \in K_n \cap L^{\perp}} \|y - \theta^2\|.
\end{equation*}
Here $L^{\perp}$ is the subspace of $\R^n$ orthogonal to $L.$ We make this claim clearer when we prove Theorem~\ref{main}. With this viewpoint, we now state the following lemma which is a direct consequence of Theorem 12 in~\cite{bellec2015sharp}, applied to the cone $K_n \cap L^{\perp}.$

\begin{lemma}[~\cite{bellec2015sharp}]\label{tech}
Let $\theta^* \in K_n \cap L^{\perp}.$ Let $s > 0$ be such that $$\E \sup_{\theta \in K_n \cap L^{\perp}: \|\theta - \theta^*\| \leq s} \langle z, \nu - \theta^* \rangle \leq \frac{s^2}{2}.$$ Then for any $x >0$ the following inequality holds with probability greater than $1 - \exp(-x),$
\begin{equation*}
\|\hat{\theta}_{K_n \cap  L^{\perp}} - \theta^*\|^2 \leq 2 \max\{s^2,8\sigma^2 x\}.
\end{equation*}
Here $\hat{\theta}_{K_n \cap  L^{\perp}}$ is the least squares estimator constrained to the cone $K_n \cap L^{\perp}.$ Also $z$ is a $n$ dimensional gaussian random vector with mean zero and covariance matrix $\sigma^2 I.$
\end{lemma}

\begin{remark}
Actually Theorem 12 in~\citet{bellec2015sharp} implies a slightly stronger bound and is useful in the case of model misspecification. This issue is discussed in Section 4.
\end{remark}

Recall the definition of the function $f_{\theta^*}$ in~\eqref{gausup}. Also recall that for any vector $\theta \in \R^n$ the range of $\theta$ is denoted by $V(\theta).$ The key step in proving Theorem~\ref{main} is to obtain the following key upper bound on $f_{\theta^*}$ which we state as a proposition:
\begin{proposition}\label{key}
Fix a positive integer $n$ and fix any $\theta^* \in K_n.$ Also fix any $t > 0.$ Then there exists a universal constant $C$ such that the following inequality holds :
\begin{equation*}
f_{\theta^*}(t) \leq C \sigma \big\{\big(V(\theta^*) + \sigma \big)^{1/4} n^{1/8} t^{3/4} + \sqrt{\log n} t\big\} + \frac{t^2}{4}.
\end{equation*}
\end{proposition}

We prove the above proposition in the next section. We first prove Theorem~\ref{main} assuming the above proposition is true.

\begin{proof}[Proof of Theorem~\ref{main}]
Let $K = K_n$ for simplicity. Fix $\theta^* \in K.$ Recall that the subspace spanned by the $n$ dimensional vectors $(1,\dots,1)$ and $(1,\dots,n)$ is denoted as $L.$ It is clear that $L$ is the smallest subspace contained in $K.$ Also recall that $P_{L}$ is the orthogonal projection matrix to the subspace $L.$ Let the projection of any $\theta \in \R^n$ onto any closed convex cone $C \subset \R^n$ be denoted by $\Pi_{C}(\theta).$ The projection exists and is unique because $C$ is a closed convex set. Now by definition of the LSE and by orthogonal decomposition of projections we have 
\begin{equation*}
\hat{\theta} = \Pi_{K} y = \Pi_{K}(\theta^* + z)  = \Pi_{K \cap L^{\perp}}(\theta^* + z)  + P_{L}(\theta^* + z).
\end{equation*}
Let $P_{L^{\perp}} = I - P_{L}$ and  $\mu^* = P_{L^{\perp}} \theta^*.$ Now we can write
\begin{align}\label{deco}
\begin{split}
&\|\hat{\theta} - \theta^*\|^2 = \|\Pi_{K \cap L^{\perp}}(\theta^* + z) - P_{L^{\perp}} \theta^*\|^2 + \|P_{L}(\theta^* + z) - P_{L} \theta^*\|^2 = \\& \|\Pi_{K \cap L^{\perp}}(\mu^* + z) - \mu^*\|^2 + \|P_{L} z\|^2.
\end{split}
\end{align}
The first equality is a sum of squares decomposition using orthogonality. The second equality is because $\Pi_{K \cap L^{\perp}}(\theta^* + z) = \Pi_{K \cap L^{\perp}}(\mu^* + z)$ by definition of $\mu^*.$ This can be checked by the usual KKT conditions for the projection of a point onto a closed convex cone. Now since $L$ is a subspace with dimension $2$ the term $\frac{\|P_{L} z\|^2}{\sigma^2}$ is distributed as a $\chi^2$ random variable with $2$ degrees of freedom. Hence by standard tail inequalities of a $\chi^2$ random variable we have the following inequality which holds for any $x > 0$ with probability greater than $1 - \exp(-x^2/16),$
\begin{equation}\label{chisq}
\|P_{L} z\|^2 \leq \sigma^2(2 + x). 
\end{equation}
Now since $\mu^* \in K \cap L^{\perp}$ and $\|\Pi_{K \cap L^{\perp}}(\mu^* + z) - \mu^*\|^2$ is exactly the squared error loss term for the least squares estimator constrained to the cone $K \cap L^{\perp}$ we can use Lemma~\ref{tech} to upper bound $\|\Pi_{K \cap L^{\perp}}(\mu^* + z) - \mu^*\|^2.$ First we use Proposition~\ref{key} to get the following inequality inequality for all $t \geq 0,$
\begin{align*}
&\E \sup_{\nu \in K_n \cap L^{\perp}: \|\nu - \mu^*\| \leq t} \langle z, \nu - \mu^* \rangle \leq f_{\mu^*}(t) \leq \\& C \sigma \big\{\big(V(\mu^*) + \sigma \big)^{1/4} n^{1/8} t^{3/4} + \sqrt{\log n} t\big\} + \frac{t^2}{4}
\end{align*}
where $C$ is a universal constant. Now it is not too hard to check that by setting $s = \big\{C \sigma \big(V(\mu^*) + \sigma\big)^{1/4} n^{1/8}\big\}^{4/5}$ for a large enough universal constant $C$ we have $f_{\mu^*}(s) < \frac{s^2}{2}.$ Setting this value of $s$ in Lemma~\ref{tech} alongwith the last display then gives us the following bound which holds for any $x > 0$ with probability greater than $1 - \exp(-x),$
\begin{equation}\label{gw}
\|\Pi_{K \cap L^{\perp}}(\mu^* + z) - \mu^*\|^2 \leq 2 \max\{C \sigma^{8/5} \big(V(\mu^*) + \sigma \big)^{2/5} n^{1/5},8\sigma^2 x\}.
\end{equation}
Using the upperbounds~\eqref{chisq} and~\eqref{gw} and combining them in~\eqref{deco} by a simple union bound argument finishes the proof for the high probability bound statement in~\eqref{highp}. To prove the risk bound in expectation, let us denote $W = \|\Pi_{K \cap L^{\perp}}(\mu^* + z) - \mu^*\|^2.$ Then we have the following inequality for any $v \geq 0$ due to~\eqref{gw}:
\begin{equation*}
P(W \geq 16 \sigma^2 v) =  \begin{cases} 1 &\mbox{if }  C \sigma^{8/5} \big(V(\mu^*) + \sigma \big)^{2/5} n^{1/5} > 16 \sigma^2 v \\ \leq \exp(-v) &\mbox{if }  C \sigma^{8/5} \big(V(\mu^*) + \sigma \big)^{2/5} n^{1/5} < 16 \sigma^2 v \end{cases}
\end{equation*}

Since $\E W = \int_{0}^{\infty} P(W \geq v) dv,$ simple integral calculus now gives us
\begin{equation*}
\E W \leq C \sigma^{8/5} \big(V(\mu^*) + \sigma \big)^{2/5} n^{1/5}
\end{equation*}
where $C$ is a universal constant. Also we have 
\begin{equation*}
\E \|P_{L} z\|^2 \leq 2 \sigma^2. 
\end{equation*}
because $\frac{\|P_{L} z\|^2}{\sigma^2}$ is a $\chi^2$ random variable with degrees of freedom $2.$ The last two displays alongwith~\eqref{deco} finish the proof of the risk bound~\eqref{expbd}.
\end{proof}

\begin{remark}\label{design}
We remark that we have defined $K$ to be the space of concave sequences obtained by evaluations of a concave function $f:[0,1] \rightarrow \R$ on equally spaced design points $(x_1,\dots,x_n) \in \R^n.$ Our risk bound calculations carry through if $(x_1,\dots,x_n)$ are not equally spaced, the only difference being that the  universal constant $C$ would now be replaced by a constant which would only depend on the ratio $\max_{i}(x_{i + 1} - x_{i})/\min_{i}(x_{i + 1} - x_{i}).$ This means that our risk bounds would continue to hold in the slightly more general situation where the gaps between consecutive design points lie between $\frac{c_{1}}{n}$ and $\frac{c_{2}}{n}$ for some constants $c_1,c_2.$ This is in the same spirit as in the risk analysis given in~\cite{GSvex}. 
\end{remark}

\subsection{\textbf{Proof of Proposition~\ref{key}}}\hspace*{\fill}\\
In order to prove Proposition~\ref{key}, we will prove several lemmas along the way. We are required to upper bound the expected Gaussian maxima function $f_{\theta^*}.$ To do this, we use a standard chaining bound as follows:

\begin{theorem}[Chaining]\label{dudthm}
Let $\F \subset \R^n$ and fix any $\theta^* \in \F.$ Let $d = \sup_{\theta,\theta^{'} \in \F} \|\theta - \theta^{'}\|$ be the diameter of $\F.$ Then we have 
\begin{equation*}\label{gaup}
\E \left[ \sup_{\theta \in \F} \left<z, \theta - \theta^* \right> \right] \leq  12 \sigma \int_{0}^{d}
      \sqrt{\log N(\epsilon,\F)} \;d\epsilon 
\end{equation*}
where $z$ is again a $n$ dimensional Gaussian random vector with mean zero and covariance matrix $\sigma^2 \I.$
\end{theorem}

We also use a standard Gaussian Concentration Inequality. The proof can be found in the argument after equation(2.35) in~\citet{Ledoux01conc}:
\begin{theorem}\label{gaussconc}[Gaussian Concentration Inequality]
Let $z$ be an n dimensional Gaussian random vector with covariance matrix $\sigma^2 \I.$ Let $f: \R^n \rightarrow \R$ be a function that is $L$ lipschtitz, that is it satisfies $|f(x) - f(y)| \leq L \|x - y\|$ for all $x$ and $y,$ where $L$ is a positive constant. Then the following is true for any $t \geq 0,$
\begin{equation*}
\P\left(f(z) \leq \E(f(z) + \sigma t\right) \leq \exp(-\frac{t^2}{2 L^2}).
\end{equation*}
\end{theorem}
 
We also need to use a log covering number bound for the space of bounded convex functions defined on the unit interval, as proved in~\citet{Dryanov}. We first set up some notations. For a metric space $\F$ with metric $D$ and $\epsilon > 0$, let $N(\epsilon, \F, D)$ denote the $\epsilon$-covering number of $\F$ under the metric $D.$ That is, $N(\epsilon, \F, D)$ is the minimum number of balls of radius $\eps$ required to cover $\F.$ If $\F \subset \R^n$ and $D$ is the usual Euclidean metric we simply denote the covering number by $N(\epsilon,\F).$
\begin{lemma}[Dryanov]\label{dryanov}
Let $C[0,1,B]$ be the space of real valued concave functions defined on the unit interval $[0,1]$ with absolute value bounded by $B$ for some $B > 0.$ Also let $L_2(f,g) = (\int_{0}^{1} (f(x) - g(x))^2 dx)^{1/2}$ for any $f,g \in C[0,1,B].$ Then the following is true for a universal constant $C,$
\begin{equation*}
\log N(\eps,C[0,1,B],L_2) \leq C \sqrt{\frac{B}{\eps}} \:\:\:\:\: \forall \eps > 0.
\end{equation*}
\end{lemma}

Modifying the above result we can now prove a log covering number bound for the set of concave sequences bounded by a number $B.$
\begin{lemma}\label{convcov}
Fix $B > 0.$ Let $K_{n,B} = \{\theta \in K_n: \max_{1 \leq i \leq n} |\theta_i| \leq B\}.$ The following is true for any $n \geq 3$ and a universal constant $C,$
\begin{equation*}
\log N(\eps,K_{n,B}) \leq C n^{1/4} \sqrt{\frac{B}{\eps}}  \:\:\:\: \forall \eps > 0.
\end{equation*}
\end{lemma}

\begin{proof}
Let $\tau = \sqrt{\frac{24}{n}} \eps.$ Let $\F$ be a finite subset of $C[0,1,B]$ such that for any $f \in C[0,1,B]$ there exists $g \in \F$ such that $L_2(f,g) \leq \tau.$ By Lemma~\ref{dryanov} $\F$ can be chosen to have $\log |\F| \leq C \sqrt{B/\tau}.$ Now take any $\theta \in K_{n,B}.$ Define a convex function $f_{\theta} \in C[0,1,B]$ as follows:
Set $f_{\theta}(\frac{i - 1}{n - 1}) = \theta_{i} \:\:\forall\:\: 1 \leq i \leq n$ and extend $f_{\theta}$ to all other points in the unit interval by linear interpolation. Clearly $f_{\theta} \in C[0,1,B].$ For each $g \in \F$ check whether there exists $\nu \in K_{n,B}$ such that $L_2(g,f_{\nu}) \leq \tau.$ If there is, choose such a $\nu$ arbitratrily and name it $\nu_{g}.$ Let $\F^{'}$ be the set of such $\nu_{g}$ obtained as we vary $g \in \F.$ Clearly we then have for a universal constant $C$,
\begin{equation}\label{card}
|\F^{'}| \leq |\F| \leq \exp(C \sqrt{B/\tau}).
\end{equation} 
Now we claim that a $2 \tau$ covering set for $K_{n,B}$ has cardinality atmost $|\F^{'}|$ which will suffice to prove the lemma.

Take any $\theta \in K_{n,B}.$ By definition of $\F$ there exists $g \in \F$ such that $L_2(f_{\theta},g) \leq \tau.$ Therefore for this $g,$ there exists $\nu_g \in \F^{'}$ such that $L_2(g,f_{\nu_g}) \leq \tau.$ Hence by the triangle inequality, we have $L_2(f_{\theta},f_{\nu_g}) \leq 2 \tau.$ Now a direct consequence of Lemma(A.4) in~\citet{GSvex} shows the following for a universal constant $C,$:
\begin{equation*}
\frac{1}{n} \|\theta - \nu_g\|^2 \leq C (L_2(f_{\theta},f_{\nu_g}))^2 \leq 4C \tau^2.
\end{equation*}
Now set $\epsilon = \sqrt{\frac{n}{4C}} \tau$ to conclude that $\F^{'}$ is a $\eps$ cover for $K_{n,B}.$ Setting the value of $\eps$ in~\eqref{card} now finishes the proof of the lemma.
\end{proof}

The space of $n$ dimensional vectors with atmost three concave blocks plays an important role in our analysis. We now define the space of sequences with atmost three concave blocks as follows. For any vector $\theta \in \R^n$ let $\theta_{[a:b]} = (\theta_{a},\dots,\theta_{b})$ for any integers $1 \leq a \leq b \leq n.$ Define
\begin{align*}
K3 = &\{\theta \in \R^n: \theta_{[1:m_1]} \in K_{m_1}, \theta_{[m_1 + 1:m_2 - 1]} \in K_{m_2 - m_1 - 1}, \\&\theta_{[m_2:n]} \in K_{n - m_2 + 1} \:\text{for some integers}\: 0 \leq m_1 < m_2 \leq n + 1\}.
\end{align*}
The interpretation is that if $m_1 = 0$ then the first block is empty. Likewise if $m_2 = n + 1$ then the third block is empty. It is not too difficult to extend Lemma~\ref{convcov} to give a bound on the log covering number of all bounded sequences in $K3$ as shown in our next lemma.
\begin{lemma}\label{3piece}
For every $\eps > 0$ we have the following:
\begin{equation*}
\log N(\eps,\{\theta \in K3 \subset \R^n: \max_{1 \leq i \leq n} |\theta_i| \leq B\}) \leq C n^{1/4} (\frac{B}{\eps})^{1/2} + 2 \log (n + 2).
\end{equation*}
\end{lemma}

\begin{proof}
Fix integers $0 \leq m_1 < m_2 \leq n + 1.$  Define $B_1 = \{1 \leq i \leq m_1\}, B_2 = \{m_1 + 1 \leq i \leq m_2 - 1\}, B_3 = \{m_2 \leq i \leq n\}.$ There are three concave pieces and it suffices to cover each of them separately at radius $\eps/\sqrt{3}$ to get an $\eps$ cover for sequences which are concave on $B_1,B_2$ and $B_3$ separately. Using Lemma~\ref{convcov} on each of these pieces one obtains a log covering number bound which is atmost $3 C n^{1/4} (\frac{B}{\tau})^{1/2}.$ Now there are exactly ${n + 2}\choose{2}$ ways of choosing $m_1$ and $m_2.$ We could cover each of the spaces defined by a fixed $m_1,m_2$ at radius $\eps$ separately and take the union of the covers. That would be a cover for $K_3$ at radius $\eps.$ The log cardinality of this cover is clearly upper bounded by $C n^{1/4} (\frac{B}{\eps})^{1/2} + \log {n + 2 \choose 2}.$ This finishes the proof of the lemma.
\end{proof}

We now embark on proving a key result which gives an upper bound on the function $f_{\theta^*}$ in case $\theta^*$ is a concave monotonic sequence. As mentioned before, since any concave function first increases and then decreases, a critical step is to understand the behaviour of $f_{\theta^*}$ when $\theta^*$ is monotonic, in addition to being concave. Recallling the definition of $f_{\theta^*}$ in~\eqref{gausup}, it is a expected supremum of Gaussian random variables where the supremum is over all $\theta \in K_n$ which also lie within a Euclidean ball around $\theta^*.$ As a first step we are only going to take the supremum over all $\theta \in K_n$ lying within a Euclidean ball around $\theta^*$ and having a maxima at a fixed index $k.$ To explain further, let us define the set $C_k$ of concave sequences with maxima at $k$ as follows:
\begin{equation*}
C_k = \{\theta \in K: \max_{1 \leq i \leq n} \theta_{i} \leq \theta_k\}.
\end{equation*}
As defined $C_k$ is a closed convex cone in $\R^n.$ Also it is clear that $K_n = \cup_{k = 1}^{n} C_k.$
Our next result is a key lemma controlling the expected Gaussian maxima term $\E \big(\sup_{\theta \in C_k: \|\theta - \theta^*\| \leq t} \langle z, \theta - \theta^* \rangle\big)$ where the supremum is taken over the restricted set $C_k$ intersected with a Euclidean ball. Our bound would hold uniformly over $k,$ as a function of $t$ whenever the underlying $\theta^*$ is concave and monotonic(non decreasing or non increasing).

\begin{lemma}\label{key1}
Fix a positive integer $n$ and fix any $1 \leq k \leq n.$ Also fix a non decreasing concave sequence $\theta^* \in K_n.$ For all $t \geq 0$ the following inequality is true for a universal constant $C,$:
\begin{equation*}
\E \big(\sup_{\theta \in C_k: \|\theta - \theta^*\| \leq t} \langle z, \theta - \theta^* \rangle\big) \leq C \sigma \big(V(\theta^*) + \sigma\big)^{1/4} n^{1/8} t^{3/4} + 2 \sigma \sqrt{2 \log(n + 2)} t + \frac{t^2}{8}.
\end{equation*}
\end{lemma}


\begin{remark}
The conclusion for Lemma~\ref{key1} works even when $\theta^*$ is concave and non increasing by reasons of symmetry.
\end{remark}

\begin{proof}
Let $A = \{\theta \in C_k: \|\theta - \theta^*\| \leq t\}.$ 
Let us define $A^{'}$ as follows:
$$A^{'} = \{\theta \in K3 \subset \R^n: \max_{1 \leq i \leq n} \theta_i \leq \theta^*_{n} + L, \min_{1 \leq i \leq n} \theta_i \geq \theta^*_{1} - L, \|\theta - \theta^*\| \leq t\}$$ where $L$ is a fixed positive number to be chosen later.

For any $\theta \in A$ we will define a truncated version of $\theta$ belonging in $A^{'}$ which will be denoted by $\theta^{'}.$ Then we will have the inequality
\begin{equation}\label{eqbrk}
\E \sup_{\theta \in A} \langle z, \theta - \theta^{*} \rangle \leq \E \sup_{\theta \in A} \langle z, \theta - \theta^{'} \rangle + \E \sup_{\theta^{'} \in A^{'}} \langle z, \theta^{'} - \theta^{*} \rangle.
\end{equation}

Fix an arbitrary $\theta \in A.$ Let us denote $S_1 = \{i: \theta_i < \theta^*_{1} - L\}$ and $S_2 = \{i: \theta_i > \theta^*_{n} + L\}.$ 

We now define $\theta^{'}$ as follows for each $1 \leq i \leq n$:
\begin{equation*}
\theta^{'}_i = \theta_{i} + (\theta^*_{1} - L)\:\:\I\{i \in S_1\} + (\theta^*_{n} + L)\:\:\I\{i \in S_2\}
\end{equation*}
where $\I$ is the usual indicator function. 
\begin{figure}
\label{fig:lemma4p2}
\begin{center}
\begin{tikzpicture}[thick, scale=2.5]
\draw [black] plot[smooth] coordinates {(-2.000000, 0.250000) (-1.800000, 0.630000) (-1.600000, 0.970000) (-1.400000, 1.270000) (-1.200000, 1.530000) (-1.000000, 1.750000) (-0.800000, 1.930000) (-0.600000, 2.070000) (-0.400000, 2.170000) (-0.200000, 2.230000) (0.000000, 2.250000) (0.200000, 2.230000) (0.400000, 2.170000) (0.600000, 2.070000) (0.800000, 1.930000) (1.000000, 1.750000) (1.200000, 1.530000) (1.400000, 1.270000) (1.600000, 0.970000) (1.800000, 0.630000) (2.000000, 0.250000) };

\draw [red] plot coordinates {(-2.000000, 0.630000) (-1.800000, 0.630000) (-1.600000, 0.970000) (-1.400000, 1.270000) (-1.200000, 1.530000) (-1.000000, 1.750000) (-0.800000, 1.930000) (-0.600000, 1.930000) (-0.400000, 1.930000) (-0.200000, 1.930000) (0.000000, 1.930000) (0.200000, 1.930000) (0.400000, 1.930000) (0.600000, 1.930000) (0.800000, 1.930000) (1.000000, 1.750000) (1.200000, 1.530000) (1.400000, 1.270000) (1.600000, 0.970000) (1.800000, 0.630000) (2.000000, 0.630000) };

\draw [dotted] plot coordinates {(-1.800000, 0.630000) (1.800000, 0.630000) };

\draw [dotted] plot coordinates {(-2.000000, 1.930000) (-0.800000, 1.930000) };

\draw [dotted] plot coordinates {(0.800000, 1.930000) (2.000000, 1.930000) };

\draw [blue] plot coordinates {(-2.000000, 0.830000) (-1.800000, 1.324352) (-1.600000, 1.397862) (-1.400000, 1.445830) (-1.200000, 1.482302) (-1.000000, 1.512072) (-0.800000, 1.537403) (-0.600000, 1.559552) (-0.400000, 1.579298) (-0.200000, 1.597158) (0.000000, 1.613496) (0.200000, 1.628574) (0.400000, 1.642592) (0.600000, 1.655706) (0.800000, 1.668035) (1.000000, 1.679679) (1.200000, 1.690717) (1.400000, 1.701217) (1.600000, 1.711234) (1.800000, 1.720814) (2.000000, 1.730000) };

\node[above] at (0,2.27) {$\theta$};

\node[above] at (0,1.95) {$\theta'$};

\node[above] at (0,1.60) {$\theta^*$};

\node[right] at (2, 1.92) {$\theta_n^*+L$};

\node[right] at (2, 1.7) {$\theta_n^*$};

\node[left] at (-2, 0.84) {$\theta_1^*$};

\node[left] at (-2, 0.63) {$\theta_1^*-L$};

\draw [dotted] plot coordinates {(0.800000, 0.100000) (0.800000, 1.930000) };

\draw [dotted] plot coordinates {(-0.800000, 0.100000) (-0.800000, 1.930000) };

\draw [dotted] plot coordinates {(-1.800000, 0.100000) (-1.800000, 0.630000) };

\draw [dotted] plot coordinates {(1.800000, 0.100000) (1.800000, 0.630000) };

\draw [gray] (-.15,.1) -- (-0.8, .1);
\draw [gray] (.15,.1) -- (0.8, .1);
\node at (0,.1) {\footnotesize $S_2$};

\node at (-1.9,.1) {\footnotesize $S_1^L$};

\node at (1.9,.1) {\footnotesize $S_1^R$};

\end{tikzpicture}

\caption{The proof of Lemma~\ref{key1} is based on
a truncation $\theta'$ (in red) of an arbitrary concave sequence $\theta\in C_k$ (in black)
with respect to a fixed monotone increasing and concave sequence $\theta^*$ (in blue). The
tails of $\theta$ are raised to $\theta_1^*-L$
over $S_1^L\cup S_1^R = S_1 = \{i : \theta_i < \theta_1^* - L\}$.
In the interval $S_2 = \{i : \theta_i > \theta_n^* + L\}$ around the
mode $k$, the sequence is lowered to the level
$\theta_n^* +L$.}
\end{center}
\end{figure}
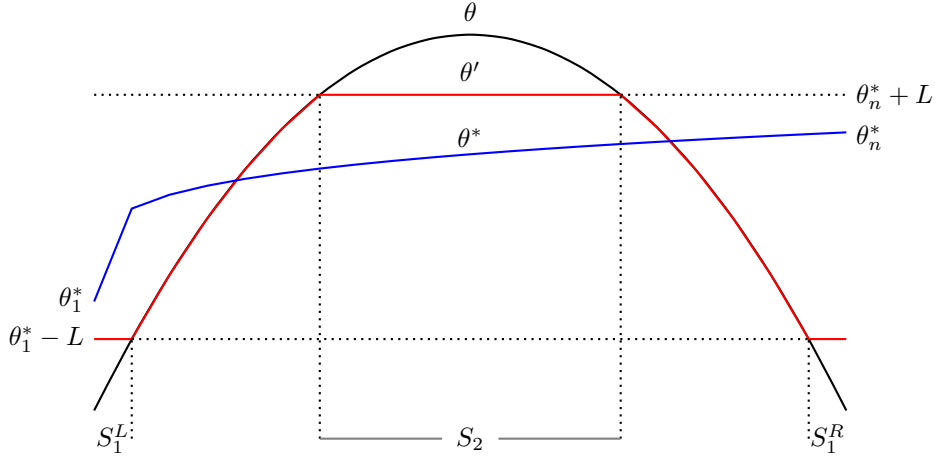

From the above definition, it is clear that $\min_{1 \leq i \leq n} \theta^{'}_{1} \geq \theta^{*}_{1} - L, \max_{1 \leq i \leq n} \theta^{'}_{i} \leq \theta^{*}_{n} + L.$ Also by construction of $\theta^{'},$ we have the following contractive property for any $1 \leq i \leq n$:
\begin{equation*}\label{contrac}
|\theta_i - \theta^{'}_i| \leq |\theta_i - \theta^{*}_i|.
\end{equation*}
Now by concavity of $\theta$, the set $S_1$ is necessarily a union of atmost two intervals of the form $\{1 \leq i \leq m_1\} \cup \{m_2 \leq i \leq n\}$ for some integers $0 \leq m_1 < m_2 \leq n + 1.$ If $m_1 = 0$ then the set $\{1 \leq i \leq m_1\}$ is empty and similarly if $m_2 = n + 1$ then the set $\{m_2 \leq i \leq n\}$ is empty. As defined, $\theta^{'}$ is a constant vector on the two intervals in $S_1$ and remains a concave vector on the complement of $S_1$ which is an interval. Hence, $\theta^{'} \in K3.$ Combining these properties of $\theta^{'}$ imply that $\theta^{'} \in A^{'}.$  

We now proceed to control the first term on the right side of the inequality in~\eqref{eqbrk}.
We have the following inequality by definition of $\theta^{'}$:
\begin{align}\label{eqtrunc}
\begin{split}
\sum_{i = 1}^{n} z_i (\theta_i - \theta'_i) 
& = \sum_{i\in S_1} z_i (\theta_i - \theta'_i) + \sum_{i\in S_2} z_i (\theta_i - \theta'_i)\\
& \leq \sum_{i\in S_1} |z_i| (\theta_i' - \theta_i) + \sum_{i\in S_2} |z_i| (\theta_i - \theta'_i)\\
& = \sum_{i \in S_1} \sum_{j = 0}^{\infty} |z_i| (\theta_i' - \theta_i) \:\:\I\{2^j L < \theta^{*}_1 - \theta_i \leq 2^{j + 1} L\}
\\
&\qquad +  \sum_{i \in S_2} \sum_{j = 0}^{\infty} |z_i| (\theta_i -
\theta'_i) \:\:\I\{2^j L < \theta_i - \theta^*_{n} \leq 2^{j + 1} L\}
 \\
& \leq \sum_{j = 0}^{\infty} 2^{j + 1} L \sum_{i \in S_1}
 |z_i|\:\:\I\{2^j L < \theta^{*}_1 - \theta_i \leq 2^{j + 1} L\} 
\\
& \qquad 
+  \sum_{j = 0}^{\infty} 2^{j + 1} L \sum_{i \in S_2} |z_i|\:\:\I\{2^j L < \theta_i - \theta^*_{n} \leq 2^{j + 1} L\}
\end{split}
\end{align}
where $\I$ denotes the indicator function and the last inequality follows from the inequalities
\begin{align}
\theta_i' - \theta_i \leq \theta_1^* - \theta_i & \;\; \text{for $i\in S_1$}\\[5pt]
\theta_i - \theta_i' \leq \theta_i - \theta_n^* & \;\; \text{for $i\in S_2$}.
\end{align}
Fix a non negative integer $j.$ We now note that for any $\theta \in A$ since $\|\theta - \theta^{*}\| \leq t$ we have 
\begin{align}\label{l2con}
\left| \bigl\{ i : 2^j L < |\theta_i - \theta_i^*| \bigr\}\right| \leq \frac{t^2}{2^{2j} L^2} \equiv v_j.
\end{align}

We now make some further observations. Since $\theta$ is concave, any set of the form $\{i: \theta_i < a\}$ for some number $a$ is necessarily atmost a union of two intervals. One interval, if non empty, has to contain the index $1$ and the other interval, if non empty has to contain the index $n.$ Hence there exists integers $0 \leq w_1 < w_2 \leq n + 1$ such that the following holds:
\begin{align*}
&\I\{i \in S_1: 2^j L < \theta^{*}_1 - \theta_i \leq 2^{j + 1} L\} \leq \I\{i \in S_1: \theta_i < \theta^{*}_i - 2^j L\} \leq \\& \I\{1 \leq i \leq w_1\} + \I\{w_2 \leq i \leq n\}.
\end{align*}
Also we have
\begin{equation*}
\sum_{i = 1}^{n} \I\{i \in S_1: 2^j L < \theta^{*}_1 - \theta_i \leq 2^{j + 1} L\} \leq \sum_{i = 1}^{n} \I\{i \in S_1: 2^j L < \theta^{*}_i - \theta_i\} \leq v_j
\end{equation*}
where the first inequality is because $\theta^*$ is non decreasing and the second inequality is due to~\eqref{l2con}. The last two displays imply the following inequality:
\begin{equation}\label{1uni}
\I\{i \in S_1: 2^j L < \theta^{*}_1 - \theta_i \leq 2^{j + 1} L\} \leq \I\{1 \leq i \leq v_j\} + \I\{n - v_j < i \leq n\}.
\end{equation}

\begin{figure}
\begin{center}
\begin{tabular}{cc}
\qquad\qquad 
\begin{tikzpicture}[thick, scale=4]
\draw [black] plot[smooth] coordinates {(-1.000000, 1.750000) (-0.800000, 1.930000) (-0.600000, 2.070000) (-0.400000, 2.170000) (-0.200000, 2.230000) (0.000000, 2.250000) (0.200000, 2.230000) (0.400000, 2.170000) (0.600000, 2.070000) (0.800000, 1.930000) (1.000000, 1.750000)};
\draw [red] plot coordinates { (-1.000000, 1.750000) (-0.800000, 1.930000) (-0.600000, 1.930000) (-0.400000, 1.930000) (-0.200000, 1.930000) (0.000000, 1.930000) (0.200000, 1.930000) (0.400000, 1.930000) (0.600000, 1.930000) (0.800000, 1.930000) (1.000000, 1.750000) };
\draw [dotted] plot coordinates {(-0.9, 2.17) (0.9, 2.17) };
\draw [blue] plot coordinates {(-1.000000, 1.512072) (-0.800000, 1.537403) (-0.600000, 1.559552) (-0.400000, 1.579298) (-0.200000, 1.597158) (0.000000, 1.613496) (0.200000, 1.628574) (0.400000, 1.642592) (0.600000, 1.655706) (0.800000, 1.668035) (1.000000, 1.679679)};
\node[above] at (0,2.27) {$\theta$};
\node[above] at (0,1.95) {$\theta'$};
\node[above] at (0,1.6) {$\theta^*$};
\node[right] at (0.9, 2.17) {$\theta_n^*+2^jL$};
\node[right] at (0.9, 1.93) {$\theta_n^*+L$};
\draw [dotted] plot coordinates {(0.400000,1.4) (0.400000, 2.17) };
\draw [dotted] plot coordinates {(-0.400000,1.4) (-0.400000, 2.17) };
\node at (0,1.3) {\footnotesize $\leq 2v_j$};
\end{tikzpicture}
\end{tabular}
\caption{The set $\bigl\{ i \in S_2 : 2^j L < \theta_i - \theta^*_n\bigr\}$
if it is nonempty, is an interval of length no greater than
$2v_j.$ The figure above indicates this set. Similarly, the set $\bigl\{i \in S_1: 2^j L < \theta^*_1 - \theta_i \bigr\}$ is
the union of at most two intervals. Each has size no larger than
$v_j$.}
\end{center}
\end{figure}
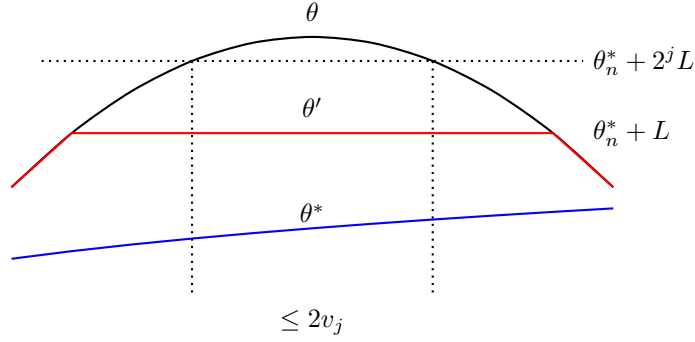

Similarly, by concavity of $\theta$ and since $\theta \in C_k$ any set of the form $\{i: \theta_i > a\}$ for some number $a$ is necessarily an interval containing $k$ if it is non empty. Hence there exists integers $1 \leq w_3 \leq k \leq w_4 \leq n$ such that the following holds:
\begin{align*}
&\I\{i \in S_2: 2^j L < \theta_i - \theta^{*}_n \leq 2^{j + 1} L\} \leq \I\{i \in S_2: \theta_i > \theta^{*}_n + 2^j L\} \leq \\& \I\{w_3 \leq i \leq w_4\}.
\end{align*}
Also we have
\begin{equation*}
\sum_{i = 1}^{n} \I\{i \in S_2: 2^j L < \theta_i - \theta^{*}_n \leq 2^{j + 1} L\} \leq \sum_{i = 1}^{n} \I\{i \in S_2: 2^j L < \theta_i - \theta^{*}_i\} \leq v_j
\end{equation*}
where the first inequality is because $\theta^*$ is non decreasing and the second inequality is due to~\eqref{l2con}. The last two displays imply the following inequality:
\begin{equation}\label{2uni}
\I\{i \in S_2: 2^j L \leq \theta_i - \theta^*_{n} \leq 2^{j + 1} L\} \leq \I\{k - v_j < i < k + v_j\}.
\end{equation}
Now using the inequalities~\eqref{1uni} and~\eqref{2uni} and applying them in~\eqref{eqtrunc} we obtain 
\begin{equation*}
\sum_{i = 1}^{n} (\theta_i - \theta^{'}_i) z_i \leq \sum_{j = 0}^{\infty} 2^{j + 1} L \left(\sum_{i = 1}^{v_j} |z_i| + \sum_{i = k - v_j + 1}^{k + v_j - 1} |z_i| + \sum_{i = n - v_j + 1}^{n} |z_i|\right).
\end{equation*}
Note that in our argument $\theta$ is an arbitrary element in $A$ and the upper bound in the previous inequality does not depend on the choice of $\theta.$ Therefore, using the fact $\E |z_i| = \sigma \sqrt{2/\pi} \leq \sigma$ alongwith~\eqref{l2con}, the last display gives us
\begin{align*}
&\E \sup_{\theta \in A} \sum_{i = 1}^{n} (\theta_i - \theta^{'}_i) z_i \leq \sum_{j = 0}^{\infty} 2^{j + 1} L \sigma (4 v_j) \leq \frac{ 4 t^2 \sigma}{L} \sum_{j = 0}^{\infty} 2^{j + 1 - 2j} =  \frac{16 t^2 \sigma}{L}.
\end{align*}
We now set $L = 128 \sigma$ to finally get
\begin{equation}\label{app2}
\E \sup_{\theta \in A} \langle z,\theta - \theta^{'} \rangle \leq \frac{t^2}{8}
\end{equation}



Now we come to controlling the second term in the right side of~\eqref{eqbrk}. Setting $\F = A^{'}$ in the chaining result in Theorem~\ref{dudthm} we get the upper bound 
\begin{equation}\label{dudcal}
\E \sup_{\theta^{'} \in A^{'}} \langle z, \theta^{'} - \theta^{*} \rangle \leq 12 \sigma \int_{0}^{2t} \sqrt{\log N(\eps,A^{'}})\;d\eps .  
\end{equation}
The upper limit of the integral is the diameter of $A^{'}$ which is atmost $2t.$ By definition of $A^{'}$ we can now apply Lemma~\ref{3piece} with $B =  \frac{V(\theta^*)}{2} + L$ to obtain $$\log N(\eps,A^{'}) \leq C n^{1/4} (\frac{B}{\eps})^{1/2} + 2 \log (n + 2).$$ Using~\eqref{dudcal} and integrating the above expression gives us 
\begin{equation*}
\E \sup_{\theta^{'} \in A^{'}} \langle z, \theta^{'} - \theta^{*} \rangle \leq C \sigma \big(\frac{V(\theta^*)}{2} + L\big)^{1/4} n^{1/8} t^{3/4} + 2 \sigma \sqrt{2 \log(n + 2)} t
\end{equation*}
where we have also used the elementary inequality $\sqrt{a + b} \leq \sqrt{a} + \sqrt{b}$ for any two positive numbers $a,b.$ Combining the last display with~\eqref{app2} and~\eqref{eqbrk} finishes the proof of the lemma.

\end{proof}

Our next step is to extend Lemma~\ref{key1} to the case when the supremum of the Gaussian inner products are taken over all of $K_n$ (not just $C_k$) intersected with a Euclidean ball. This result is presented in our next lemma.
\begin{lemma}\label{key2}
Fix a positive integer $n.$ Also fix a concave sequence $\theta^*$ which is non decreasing sequence or non increasing. For all $t \geq 0$ the following inequality is true for a universal constant $C,$:
\begin{equation*}
\E \big(\sup_{\theta \in K_n: \|\theta - \theta^*\| \leq t} \langle z, \theta - \theta^* \rangle\big) \leq C \sigma \big\{V(\theta^*) + \sigma \big)^{1/4} n^{1/8} t^{3/4} + \sigma \sqrt{\log n} t\big\} + \frac{t^2}{8}.
\end{equation*}
\end{lemma}

\begin{proof}
We prove the lemma when $\theta^*$ is a concave non decreasing sequence. The proof when $\theta^*$ is concave non increasing is analogous. For each $1 \leq k \leq n$ and $t > 0$ define the random variables
\begin{equation*}
X_k(t) = \sup_{\theta \in \C_k: \|\theta - \theta^*\| \leq t} \langle z, \theta - \theta^* \rangle.
\end{equation*}

We first note that 
\begin{equation*}
\sup_{\theta \in K_n: \|\theta - \theta^*\| \leq t} \langle z, \theta - \theta^* \rangle = \max_{1 \leq k \leq n} X_k(t).
\end{equation*}
Applying Lemma~\ref{gaulip} (see appendix) we see that the random variables $X_k(t)$ are lipschitz functions of $z$ with lipschitz constant $t.$ Hence using the Gaussian Concentration Theorem~\ref{gaussconc} for lipschitz functions we get for all $x > 0$ and all $1 \leq k \leq n,$
\begin{equation*}
P\big(X_k(t) \leq \E X_k(t) + t \sigma x\big) \leq \exp(-\frac{x^2}{2}).
\end{equation*}
A standard argument involving maxima of random variables with Gaussian like tails is given in Lemma~\ref{subg} for the sake of completeness. Using this lemma and the last display we finally get for a universal constant $C,$
\begin{equation*}
\E \max_{1 \leq k \leq n} X_k(t) \leq \max_{1 \leq k \leq n} \E X_k(t) + C \sigma t \sqrt{\log n}.
\end{equation*}
Now Lemma~\ref{key1} gives us an upper bound on the term $\max_{1 \leq k \leq n} E X_k(t)$ because the upper bound in Lemma~\ref{key1} does not depend on $k.$ Using this upper bound alongwith the last display finishes the proof of the lemma. 
\end{proof}

We are now finally ready to prove Proposition~\ref{key}. The main idea is to use the fact that any concave sequence first increases and then decreases and use Lemma~\ref{key2}.

\begin{proof}[Proof of Proposition~\ref{key}]
Let $1 \leq i^* \leq n$ be such that $\theta^* \in C_{i^*}.$ Let $\theta^* = (\theta^*_1,\theta^*_2)$ where $\theta^*_1$ is an $i^*$ dimensional vector and $\theta^*_2$ is an $n - i^*$ dimensional vector. Similarly, let $z = (z_1,z_2).$ We can write
\begin{equation*}
\E \big(\sup_{\theta \in K_n: \|\theta - \theta^*\| \leq t} \langle z, \theta - \theta^* \rangle\big) \leq \E \big(\sup_{\theta \in \K_{i^*}: \|\theta - \theta^*_1\| \leq t} \langle z_1, \theta - \theta^*_1 \rangle\big) +  \E \big(\sup_{\theta \in \K_{n - i^*}: \|\theta - \theta^*_2\| \leq t} \langle z_2, \theta - \theta^*_2 \rangle\big).
\end{equation*}

There are two terms on the right side of the above inequality. Let us first bound the first term on the right side. The second term can be bounded exactly in the same way. Since $\theta_1^{*}$ is non decreasing, we can use Lemma~\ref{key2} to obtain for all $t \geq 0$:
\begin{equation*}
\E \big(\sup_{\theta \in \K_{i^*}: \|\theta - \theta^*_1\| \leq t} \langle z_1, \theta - \theta^*_1 \rangle\big) \leq C \sigma \big\{\big(V(\theta^*) + \sigma\big)^{1/4} n^{1/8} t^{3/4} + \sqrt{\log n} t\big\} + \frac{t^2}{8}.
\end{equation*}
where we also use the fact that $(i^*)^{1/8} \leq n^{1/8}.$ We upper bound the second term by using Lemma~\ref{key2} exactly similarly. We then obtain for an appropriate universal constant $C,$ 
\begin{equation*}\label{eqstep}
f_{\theta^*}(t) = \E \big(\sup_{\theta \in K_n: \|\theta - \theta^*\| \leq t} \langle z, \theta - \theta^* \rangle\big) \leq C \sigma \big\{\big(V(\theta^*) + \sigma \big)^{1/4} n^{1/8} t^{3/4} + \sqrt{\log n} t\big\} + \frac{t^2}{4}.
\end{equation*}
This finishes the proof.
\end{proof}

\section{Model Misspecification}\label{secmisspec}
Our risk analysis actually extends to the case of model misspecification. A natural quantity in the misspecified case for measuring the performance of an estimator $\tilde{\theta}$ is to control what is called the regret,  defined as $\|\tilde{\theta} - \theta^*\|^2 - \min_{\theta \in K_n} \|\theta - \theta^*\|^2.$ The goal of this section is to prove the next theorem upper bounding the regret of the LSE. This theorem gives an oracle risk bound generalizing Theorem~\ref{main} to the case when the true underlying sequence $\theta^*$ is not necessarily concave. Recall that $\Pi_{C}$ denotes the projection operator to any closed convex cone $C$ and for any vector $\theta \in \R^n$ the range of the vector is denoted by $V(\theta) = \max_{1 \leq i \leq n} \theta_i - \min_{1 \leq i \leq n} \theta_i.$

\begin{theorem}\label{misspec}
Fix any positive integer $n.$ Fix any $\theta^* \in \R^n.$ Let $K = K_n$ for simplicity. Let $\mu^* = (I - P_{L}) \theta^*.$ Let $H(\theta^*) = V(\Pi_{K}(\mu^*)).$ Fix any $x > 0.$ The following upper bound on the risk holds for a universal constant $C$ with probability greater than $1 - \exp(-x) - \exp(-x^2/16):$
\begin{align*}\label{highp}
\begin{split}
\frac{1}{n}\|\hat{\theta} - \theta^*\|^2 \leq \:\:&\frac{1}{n}\|\Pi_{K \cap L^{\perp}}(\mu^*) - \mu^*\|^2 + \frac{1}{n} \sigma^2(2 + 17x) \:+ \\& C \sigma^{8/5} \big(H(\theta^*) + \sigma \big)^{2/5} n^{-4/5}.
\end{split}
\end{align*}
Moreover, the above high probability risk bound also immediately implies a risk bound in expectation for a universal constant $C$:
\begin{equation*}\label{expbd}
R(\hat{\theta},\theta^*) \leq \frac{1}{n}\|\Pi_{K \cap L^{\perp}}(\mu^*) - \mu^*\|^2 + C \sigma^{8/5} \big(H(\theta^*) + \sigma \big)^{2/5} n^{-4/5} + \frac{2 \sigma^2}{n}.
\end{equation*}
\end{theorem}

The above theorem implies that in case the true underlying sequence $\theta^*$ is non concave then the regret of the LSE converges at the rate $n^{-4/5}$ upto constant factors.

Our starting point in proving the above theorem is actually a stronger version of Lemma~\ref{tech} and is a direct consequence of Theorem 12 in~\citet{bellec2015sharp} when applied to the cone $K_n \cap L^{\perp} \subset K_n.$

\begin{lemma}[Bellec]\label{tech2}
Let $\theta^* \in \R^n$ and let $K = K_n.$ Let $\mu^* = (I - P_{L}) \theta^*.$ Let $s > 0$ be such that $f_{\Pi_{K}(\mu^*)}(s) \leq \frac{s^2}{2}.$ Then for any $x >0$ the following inequality holds with probability greater than $1 - \exp(-x),$
\begin{equation*}
\|\Pi_{K \cap L^{\perp}}(\mu^* + z) - \mu^*\|^2 \leq \|\mu^* - \Pi_{K \cap L^{\perp}}(\mu^*)\|^2 + 2 \max\{s^2,8\sigma^2 x\}.
\end{equation*}
Here $z$ is a $n$ dimensional gaussian random vector with mean zero and covariance matrix $\sigma^2 I.$
\end{lemma}

We now prove Theorem~\ref{misspec} which can be done in exactly the same way as the proof of Theorem~\ref{main}.
\begin{proof}[Proof of Theorem~\ref{misspec}]
Let $K = K_n$ for simplicity. Let $\theta^* \in \R^n.$ Recall that the subspace spanned by the $n$ dimensional vectors $(1,\dots,1)$ and $(1,\dots,n)$ is denoted as $L.$ It is clear that $L$ is the smallest subspace contained in $K.$ Also recall that $P_{L}$ is the orthogonal projection matrix to the subspace $L.$ By definition of the LSE and by orthogonal decomposition of projections we have 
\begin{equation*}
\hat{\theta} = \Pi_{K \cap L^{\perp}}(\theta^* + z)  + P_{L}(\theta^* + z).
\end{equation*}
Let $P_{L^{\perp}} = I - P_{L}.$ Since $\mu^* = P_{L^{\perp}} \theta^*$ we can write
\begin{align*}
\begin{split}
&\|\hat{\theta} - \theta^*\|^2 = \|\Pi_{K \cap L^{\perp}}(\theta^* + z) - P_{L^{\perp}} \theta^*\|^2 + \|P_{L}(\theta^* + z) - P_{L} \theta^*\|^2 = \\& \|\Pi_{K \cap L^{\perp}}(\mu^* + z) - \mu^*\|^2 + \|P_{L} z\|^2.
\end{split}
\end{align*}
The first equality is a sum of squares decomposition using orthogonality. The second equality is because $\Pi_{K \cap L^{\perp}}(\theta^* + z) = \Pi_{K \cap L^{\perp}}(\mu^* + z)$ by definition of $\mu^*.$ This can be again checked by the usual KKT conditions for the projection of a point onto a closed convex cone. Now the term $\|\Pi_{K \cap L^{\perp}}(\mu^* + z) - \mu^*\|^2$ can be controlled by an application of Lemma~\ref{tech2} and Proposition~\ref{key} which can be used since $\Pi_{K}(\mu^*) \in K.$ The rest of the proof goes through verbatim as in the proof of Theorem~\ref{main}. 
\end{proof}

\section*{Acknowledgements}
The author thanks Adityanand Guntuboyina, John Lafferty, Bodhisattva Sen and Yuancheng Zhu for helpful discussions.

\bibliographystyle{plainnat}
\bibliography{convrefer}

\appendix
\section*{Appendix}
\begin{lemma}\label{subg}
Let $X_1,X_2\dots,X_n$ be random variables such that the following holds for every $1 \leq i \leq n$ and $a >0,$
\begin{equation}\label{appe1}
P(X_i \geq \E X_i + ax) \leq \exp(-\frac{x^2}{2}) \:\:\: \forall x \geq 0.
\end{equation}
Then the following is true:
\begin{equation*}
\E \max_{1 \leq i \leq n} X_i \leq \max_{1 \leq i \leq n} \E X_i + 2 a \left(\sqrt{2 \log n} + \sqrt{2 \pi}\right).
\end{equation*}
\end{lemma}

\begin{proof}
Let $m = \max_{1 \leq i \leq n} \E X_i.$ Define $Y_i = X_i - m.$ Then by~\eqref{appe1}
we have sub gaussian tail behaviour for every $1 \leq i \leq n,$
\begin{equation}\label{ybd}
P(Y_i \geq x) \leq \exp(-\frac{x^2}{2 a^2}) \:\:\: \forall x \geq 0.
\end{equation}
Now we have
\begin{align*}
&\E \max Y_i = \int_{-\infty}^{\infty} P(\max Y_i \geq x) dx \leq \int_{-\infty}^{\infty} \min\{n \:\:\exp(-\frac{x^2}{2 a^2}),1\} dx = \\& 2 \int_{0}^{\infty} \min\{n \:\:\exp(-\frac{x^2}{2 a^2}),1\} \leq 2 \int_{0}^{a \sqrt{2 \log n}} 1 dx + 2 \int_{a \sqrt{2 \log n}}^{\infty} n \:\:exp(-\frac{x^2}{2 a^2}) dx 
\end{align*}
where the first inequality follows from a simple union bound argument alongwith~\eqref{ybd}. The rest follows from simple integral calculus. Now a standard fact about Gaussian tails give us the following inequality
\begin{equation*}
\int_{a \sqrt{2 \log n}}^{\infty} \:\:\exp(-\frac{x^2}{2 a^2}) \leq \sqrt{2 \pi} \frac{a}{n}.
\end{equation*}
The last two displays finish the proof of the lemma.

\end{proof}

\begin{lemma}\label{gaulip}
Let $A \subset \R^n$ be a closed set. Fix any $\theta^* \in \R^n$ and $t > 0.$ Define the function $f:\R^n \rightarrow \R$ as follows:
\begin{equation*}
f(z) = \sup_{\theta \in A: \|\theta - \theta^*\| \leq t} \langle z,\theta - \theta^* \rangle.
\end{equation*}
Then $f$ is a Lipschitz function of $z$ with Lipschitz constant $t.$
\end{lemma}

\begin{proof}
Since the function $\langle z,\theta - \theta^* \rangle$ is a continuous function of $z$ and the set $\{\theta \in A: \|\theta - \theta^*\| \leq t\}$ is a compact set, its supremum is attained at $\tilde{\theta}$ say. Then we have 
\begin{align*}
\begin{split}
&f(z) - f(z^{'}) = \langle z,\tilde{\theta} - \theta^* \rangle - f(z^{'}) \leq \langle z,\tilde{\theta} - \theta^* \rangle - \langle z^{'},\tilde{\theta} - \theta^* \rangle = \langle z - z^{'},\tilde{\theta} - \theta^* \rangle \leq \\&
\|z - z^{'}\| \|\tilde{\theta} - \theta^* \| \leq t \|z - z^{'}\|.
\end{split}
\end{align*}
The first inequality is because we set $\theta = \tilde{\theta}$ instead of taking supremum over $\theta,$ the second inequality is just the Cauchy Schwarz inequality and the last inequality follows from the fact that $\|\tilde{\theta} - \theta^*\| \leq t$ by the choice of $\tilde{\theta}.$ This finishes the proof of the lemma.
\end{proof}

\end{document}